\documentclass[12pt, reqno]{amsart}
\usepackage{amsmath, amsthm, amscd, amsfonts, amssymb, graphicx, color}
\usepackage[bookmarksnumbered, colorlinks, plainpages]{hyperref}
\usepackage{float}
\usepackage{graphicx,epstopdf}
\usepackage{bbm}
\usepackage{mathrsfs}
\usepackage{dsfont}
\usepackage{amssymb}
\usepackage{csquotes}
\usepackage{latexsym}
\usepackage{exscale}
\usepackage[latin1]{inputenc}
\usepackage[all]{xy}
\usepackage{rotating}
\usepackage{amsfonts}
\usepackage{amssymb}
\usepackage{amsmath}
\usepackage{cite}
\usepackage{bbm}
\usepackage{graphicx}

\usepackage{pstricks,pst-math,pst-infixplot}
\usepackage[dotinlabels]{titletoc}
\hypersetup{colorlinks=true,linkcolor=red, anchorcolor=green, citecolor=cyan, urlcolor=red, filecolor=magenta, pdftoolbar=true}

\newtheorem{theorem}{Theorem}[section]
\newtheorem{lemma}[theorem]{Lemma}
\newtheorem{proposition}[theorem]{Proposition}

\theoremstyle{definition}
\newtheorem{definition}[theorem]{Definition}
\newtheorem{example}[theorem]{Example}

\theoremstyle{remark}

\newcommand{\R}{\mathbb{R}}
\numberwithin{equation}{section}
\begin{document}
\setcounter{page}{1}

\title{ISOMETRIC EMBEDDINGS OF FINITE METRIC TREES INTO $(\mathbb{R}^n,d_{1})$ and $(\mathbb{R}^n,d_{\infty})$}
\author[Asuman  G\"{u}ven  AKSOY, Mehmet Kili\c{c} and Sah\.{i}n Ko\c{c}ak]{Asuman G\"{u}ven AKSOY,$^1*$ Mehmet Kili\c{c}$^1$, \c{S}ah\.{I}n Ko\c{c}ak $^2$}

\address{$^{1^*}$Department of Mathematics, Claremont McKenna College, 850 Columbia Avenue, Claremont, CA  91711, USA.}
\email{\textcolor[rgb]{0.00,0.00,0.84}{aaksoy@cmc.edu}}

\address{$^{1}$ Department of Mathematics, Anadolu University, 26470, Eskisehir Turkey.}
\email{\textcolor[rgb]{0.00,0.00,0.84}{kompaktuzay@gmail.com}}

\address{$^{2}$ Department of Mathematics, Anadolu University, 26470, Eskisehir Turkey.}
\email{\textcolor[rgb]{0.00,0.00,0.84}{skocak@anadolu.edu.tr}}

%\dedicatory{This paper is dedicated to Professor ABCD}



\begin{abstract}
We investigate isometric embeddings of finite metric trees into $(\mathbb{R}^n,d_{1})$ and $( \mathbb{R}^n, d_{\infty})$.
We prove that a finite metric tree can be isometrically embedded into $(\mathbb{R}^n,d_{1})$ if and only if the number of its leaves is at most $2n$. We show that a finite star tree with at most $2^n$ leaves can be isometrically embedded into $(\mathbb{R}^{n}, d_{\infty})$ and a finite metric tree with more than $2^n$ leaves cannot be isometrically embedded into $(\mathbb{R}^{n}, d_{\infty})$.

We conjecture that an arbitrary finite metric tree with at most
$2^n$ leaves can be isometrically embedded into $(\mathbb{R}^{n},
d_{\infty})$.
\end{abstract}
 \maketitle

\section{Introduction}
A finite metric tree is a finite, connected and positively weighted graph without cycles. The distance between two vertices is given by the total weight of a simple path (which is unique) between these vertices. For two in-between points of the tree one takes the corresponding portions of the edges carrying these points into account. A leaf of a tree is a vertex with degree $1$. Vertices other than the leaves are called interior vertices. A tree with a single interior vertex is called a star (or star tree). We will not allow interior vertices of degree 2 since, from the point of view of metric properties, they can be considered artificial.

Embedding new spaces into more familiar ones is a kind of innate behavior for  mathematicians. For metric spaces, as the first candidate for an ambient space, the Euclidean space $(\mathbb{R}^n,d_{2})$ might come into mind; but it is a complete disappointment. No metric tree (with at least three leaves) can be embedded isometrically into any $(\mathbb{R}^n,d_{2})$. For a star with three leaves, this is almost obvious; and any tree with more than three leaves contains a three-star as a subspace.

There is no help in considering $(\mathbb{R}^n,d_{p})$ with any $1<p<\infty$, because these are also uniquely geodesic spaces and they do not host any tree with at least three leaves either.

At this point we could take refuge in Kuratowski's embedding theorem which states that every metric space $M$ embeds isometrically in the Banach space $L^{\infty} (M)$  of bounded functions on $M$ with sup norm $||f||_{\infty}:= \displaystyle \sup_{x\in M} |f(x)|$ where  $f: M \to \mathbb{R}$. (see \cite{Hei}). But a metric tree has an uncountable number of points so that our ambient space would be too huge and useless. A remedy could arise from considering the leaves of the tree only, since they determine the whole of the metric tree (as a metric space and up to isometry) as their tight span (see  Theorem $8$ in  Dress \cite{dre})
. By this approach we get at least an isometric embedding of the metric tree into $(\mathbb{R}^n, d_{\infty})$ where $n$ is the number of leaves and $d_{\infty}$ denotes the maximum metric $d_{\infty}(x,y)=\displaystyle \max_{1\leq i \leq n} |x_i-y_i|$.

We will show that a metric star tree can be embedded into
$(\mathbb{R}^n,d_{\infty})$ if and only if it has at most $2^n$
leaves and an arbitrary finite metric tree with more than $2^n$
leaves cannot be embedded into $(\mathbb{R}^n,d_{\infty})$. We
guess that an arbitrary finite metric tree with at most $2^n$
leaves can be embedded into $(\mathbb{R}^n,d_{\infty})$ but we are
yet unable to provide a proof for this guess.

%and hope actually that there is in fact room in $(\mathbb{R}^n, d_{\infty})$ for trees with at most $2^n$ leaves.

The picture for $(\mathbb{R}^n,d_{1})$ as the ambient space is more clear-cut. It was already shown by Evans (\cite{Ev}) that any finite metric tree can be isometrically embedded into $l_1$,
 without explicit bounds for the dimension of the target in terms of the leaf number of the given tree.
 We prove by other, more geometric means that a finite metric tree can be embedded isometrically into $(\mathbb{R}^n,d_{1})$ if and only if the number of its leaves is at most $2n$.

\section{Preliminaries}
For the sake of clarification,  we give in the following the formal definition of a metric tree and mention some of its properties. However, in this paper we will consider a special subfamily of metric trees, namely finite  connected weighted graphs without loops. Our aim is to investigate whether we can isometrically embed these finite metric trees  into $(R^n, d_1)$ or $(R^n, d_{\infty})$.
\begin{definition} \label{def21}
%\begin{definition}\label{D:mseg}
    Let $x, y \in M$, where ($M$, $d$) is a metric space.
    A \emph{geodesic segment from $x$ to $y$} (or a \emph{metric segment},
denoted by $[x,y]$) is the image of an isometric embedding
$\alpha : [a,b]\rightarrow M$ such that $\alpha(a)=x$ and $\alpha(b)=y$.
% The geodesic segment will be called a \emph{metric segment} and denoted by
% $[x,y]$ throughout this paper.
A metric space is called {\it geodesic} if any two points can be connected
by a metric segment.
% These classical definitions can be found, for instance, in \cite{Blum}.
   % \end{defi}

%\begin{definition}\label{D:mt2}
%\begin{defi}
     A metric space $(M,d)$, is called a {\it metric tree}  if and only if for all $x,y,z \in
    M$, the following holds:
    \begin{enumerate}
        \item there exists a unique metric segment from $x$ to $y$,
        and
        \item $[x,z] \cap [z,y] = \{z\} \Rightarrow [x,z] \cup [z,y] = [x,y]$.
    \end{enumerate}
\end{definition}

%%%%%%%%%%%%%%%%
Next we mention some useful properties of metric segments  that we will use. For the proofs of these properties we refer the reader to consult \cite{Blum} and \cite{Ev}.
. % to be use throughout the paper.

For $x, y$ in a metric space $M$, write  $xy = d(x,y)$.
% In the following we list some of the properties of metric trees
% which will be used throughout out this paper.%\\
For $x, y, z \in M$, we say  $y$ is \emph{between $x$ and $z$}, denoted
    $xyz$, if and only if $xz = xy + yz$.
%The following facts will be used throughout the paper:
% We also set  $$d(x,y) = xy. $$ Following will be used throughout the paper:
 \begin{enumerate}
 \item (Transitivity of betweenness \cite{Blum}) Let $M$ be a metric space and let % \\
$a,b,c,d \in M$. If $abc$ and $acd$, then $abd$ and $bcd$.
 \item  (Three point property,  \cite[Section 3.3.1]{Ev})
Let $x,y,z \in T$ ($T$ is a complete metric tree).
There exists (necessarily unique) $w \in T$ such that $$[x,z] \cap [y,z] = [w,z]\,\,\,
       \mbox{and}\,\,\, [x,y] \cap [w,z] = \{w\} $$
       Consequently, $$[x,y] = [x,w] \cup [w,y],\,\,\,[x,z] = [x,w] \cup [w,z],\,\,\,
       \mbox{and}\,\,\, [y,z] = [y,w] \cup [w,z].$$
\end{enumerate}

%%%%%%%%%%%%%%%%%
Here are two examples of metric trees:
\begin{example}(The Radial Metric)\label{E:radial}\\
    Define $d: \R^2 \times \R^2 \to \R_{\geq 0}$ by:
    \[
        d(x,y) =
            \begin{cases}
            \|x-y\| & \text{if $x = \lambda \, y$ for some $\lambda \in \R$,}\\
            \|x\| + \| y \| & \text{otherwise.}
            \end{cases}
    \]
    It is easy to verify  that  $d$ is in fact a metric and that $(\R^2,d)$ is a metric tree.
\end{example}
\begin{example}
( ``Star Tree'')\label{E:tripod}
%In the following, we shall consider a subtree of the radial
%tree described above.

 Fix $k \in \mathbb{N}$, and a sequence of positive
numbers $(a_i)_{i=1}^k$, the {\it metric star tree} is defined
as a union of $k$ intervals of lengths $a_1, \ldots, a_k$, emanating
from a common center and equipped with the radial metric. More precisely,
our tree $T$ consists of its center $o$, and the points $(i,t)$, with
$1 \leq i \leq k$ and $0 < t \leq a_i$. The distance $d$ is defined
by setting $d(o, (i,t)) = t$, and
$$
d((i,t),(j,s)) = \left\{ \begin{array}{ll}
   |t-s|   &   i=j   \cr
   t+s     &   i \neq j
\end{array} \right. .
$$
Abusing the notation slightly, we often identify $o$ with $(i,0)$. The leaves of this metric star tree are the points $(i,a_i),\, i=1,\dots,k.$ \end{example}

 In the following we will consider only a very special and simple type of metric trees called finite metric trees or finite simplicial metric trees. They are explained, for example,
  in \cite{papa}, p.43 and p.73 (see also \cite{Ev}, Example 3.16). Topologically they are finite trees in the sense of graph theory: There is a finite set of ``vertices''; between any pair of vertices there is
   at most one ``edge'' and the emerging graph is connected and has no loops. The degree of a vertex is defined in the graph-theoretical sense and a vertex with degree 1 is called a leaf.

 In addition to this graph-theoretical structure, non-negative weights are assigned to the edges. Moreover, to a pair of vertices a distance is assigned by considering the unique simple path (in graph-theoretical sense)
  connecting these vertices and adding up the weights of the edges constituting this path. This distance can be naturally extended to any pair of points of the tree and converts the graph-theoretical tree into a metric space
   which is a metric tree in the sense of Definition \ref{def21}.

 These special metric trees are also called finite metric trees, or finite simplicial metric trees, though they obviously contain a continuum of points as a metric space (except in the trivial case of a singleton).
 We will be concerned with embedding these finite metric trees into $(\mathbb{R}^n,d_1)$ or $(\mathbb{R}^n,d_{\infty})$ where we will try to optimize the dimension $n$ in dependence of the number of the leaves of the tree.

 In general metric trees are more complicated than finite simplicial metric trees.  For further discussion of non-simplicial trees and construction of metric trees related to the asymptotic geometry of hyperbolic metric spaces we refer the reader to \cite{Brid}.
 %In this paper we will consider a special subfamily of metric trees, namely finite  connected graphs without loops. By giving weight to the edges, we obtain a specific class metric spaces.
\section{Embedding Star Trees}

As motivational examples we consider first finite metric star trees. They can be viewed  as a union of  $k$ intervals of lengths $a_1,a_2,\ldots,a_k$, emanating from a common center and equipped with the radial metric, as described above (see Fig~\ref{st1}).
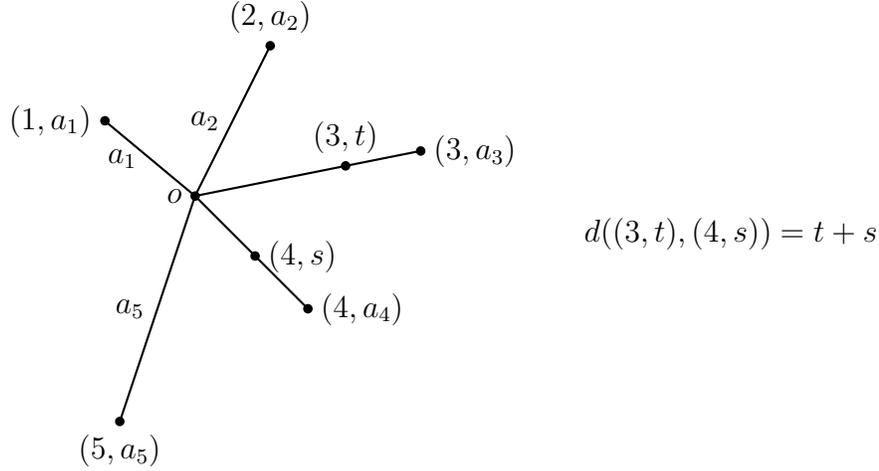
\begin{figure}[h!]
\begin{center}
\begin{pspicture*}(-3,-5)(10,3)
\uput[l](0,0){$o$}
\psline(-1,-3)(0,0)(-1.2,1)
\psline(1,2)(0,0)(3,0.6)
\psline(0,0)(1.5,-1.5)
\psdots(-1,-3)(0,0)(-1.2,1)(1,2)(1.5,-1.5)(3,0.6)(2,0.4)(0.8,-0.8)
\uput[d](-1,-3){$(5,a_5)$}
\uput[l](-1.2,1){$(1,a_1)$}
\uput[u](1,2){$(2,a_2)$}
\uput[r](3,0.6){$(3,a_3)$}
\uput[r](1.5,-1.5){$(4,a_4)$}
\uput[l](-0.5,-1.5){$a_5$}
\uput[l](-0.6,0.5){$a_1$}
\uput[l](0.5,1){$a_2$}
\uput[u](2,0.4){$(3,t)$}
\uput[r](0.8,-0.8){$(4,s)$}
\uput[r](5,-0.5){$d((3,t),(4,s))=t+s$}
\end{pspicture*}
\caption{A metric star tree.}
\label{st1}
\end{center}
\end{figure}

The following two properties give the tight embedding bounds for finite metric star trees.

\begin{proposition} \label {a0}
A metric star tree $(X,d)$ with $k$ leaves can be isometrically embedded into $(\mathbb{R}^n, d_{\infty})$ if and only if $k\leq2^n$.
\end{proposition}

\begin{proof}
($\Leftarrow$)
First assume $k\leq 2^n$. Note that there are $2^n$ extreme points on the unit ball of the space $(\mathbb{R}^n,d_{\infty})$, which are $\epsilon \in \{-1,1 \}^n$. We will denote these extreme points by $v_1,v_2,\ldots,v_{2^n}$ and define the map $f:X\rightarrow\mathbb({R}^n,d_{\infty})$ by $f(o)=O\in\mathbb{R}^n$ and $f((i,t))= t\cdot v_i$. The map $f$ is then an isometric embedding from $X$ to $(\mathbb{R}^n,d_{\infty})$:
$$d_{\infty}(f(i,t),f(o))=d_{\infty}(t\cdot v_i,O)=t=d((i,t),o),$$
$$d_{\infty}(f(i,t),f(i,s))=d_{\infty}(t\cdot v_i,s\cdot v_i)=|t-s|=d((i,t),(i,s)),$$
$$d_{\infty}(f(i,t),f(j,s))=d_{\infty}(t\cdot v_i,s\cdot v_j)=t+s=d((i,t),(j,s)), \,\,\,\mbox{when}\,\,\,  i\neq j.$$

($\Rightarrow$)
Assume that $k>2^n$ but there exists an isometric embedding $f$ from $X$ to $(\mathbb{R}^n,d_{\infty})$. We can assume that the embedding takes the center of $X$ to the origin of $\mathbb{R}^n$ because translation is an isometry. Then, we can find two points $(i,a_i)$ and $(j,a_j)$ in $X$ with $i\neq j$ such that all corresponding coordinates of their images have the same sign. If their images are $f(i,a_i)=A_i=(A_i^1,A_i^2,\cdots,A_i^n)$ and $f(j,a_j)=A_j=(A_j^1,A_j^2,\cdots,A_j^n)$, then $|A_i^m|\leq a_i$ and $|A_j^m|\leq a_j$ for all $m=1,2,\ldots,n$ since $d_{\infty}(A_i,O)=d((i,a_i),o)=a_i$ and $d_{\infty}(A_j,O)=d((j,a_j),o)=a_j$. Hence, we obtain
\[
d_{\infty}(A_i,A_j)=\max_{m=1}^{n}\{|A_i^m-A_j^m|\}<a_i+a_j=d((i,a_i),(j,a_j)),
\]
which contradicts the assumption that $f$ is an isometry.
\end{proof}

\begin{proposition}
\label{a20}
 A metric star tree $X$ with $k$ leaves can be embedded isometrically into $(\mathbb{R}^n,d_1)$ if and only if $k\leq 2n$.
\end{proposition}

\begin{proof}
($\Leftarrow$)
First assume $k\leq 2n$. Note that there are $2n$ extreme point on the unit ball of the space $(\mathbb{R}^n,d_1)$, which are $ e_i=(\delta_{i,j})_{j=1}^{n}$, where $\delta_{i,j}$ is the Kronecker delta. We will denote these extreme points in $\mathbb{R}^n$ by $E_1,E_2,\ldots,E_{2n}$ and define the map $f:X\rightarrow(\mathbb{R}^n,d_1)$ by $f(o)=O$ and $f((i,t))= t\cdot E_i$. We will show that $f$ is an isometric embedding from $X$ to $(\mathbb{R}^n,d_1)$:
$$d_1(f(i,t),f(o))=d_1(t\cdot E_i,O)=t=d((i,t),o),$$
$$d_1(f(i,t),f(i,s))=d_1(t\cdot E_i,s\cdot E_i)=|t-s|=d((i,t),(i,s)),$$
$$d_1(f(i,t),f(j,s))=d_1(t\cdot E_i,s\cdot E_j)=t+s=d((i,t),(j,s)),  \,\,\mbox{where}\,\, i\neq j .$$

($\Rightarrow$)
Assume that $k>2n$ but there exists an isometric embedding $f$ from $X$ to $(\mathbb{R}^n,d_1)$. We can assume that the embedding takes the center of $X$ to the origin because translation is an isometry. Then, we can find two points $(i,a_i)$ and $(j,a_j)$ in $X$ with $i\neq j$ such that at least one common coordinate of their images are nonzero and have the same sign. If their images are $f(i,a_i)=A_i=(A_i^1,A_i^2,\cdots,A_i^n)$ and $f(j,a_j)=A_j=(A_j^1,A_j^2,\cdots,A_j^n)$, then we obtain
\begin{eqnarray*}
d((i,a_i),(j,a_j))&=&d((i,a_i),o)+d(o, (j,a_j))=d_1(A_i,0)+d_1(0,A_j)\\
&=&|A_i^1|+|A_i^2|+\cdots+|A_i^n|+|A_j^1|+|A_j^2|+\cdots+|A_j^n| \\
 &>& |A_i^1-A_j^1|+|A_i^2-A_j^2|+\cdots+|A_i^n-A_j^n|\\
 &=&d_1(A_i,A_j),
\end{eqnarray*}
which is a contradiction.
\end{proof}

\section{Embedding Arbitrary Metric Trees}

For later use, we recall some metric preliminaries.

\begin{definition}
For $p=(p_1,p_2,\ldots,p_n)\in( \mathbb{R}^n, d_{\infty})$, we define
\begin{eqnarray*}
S_i^{+}(p)&=&\{q=(q_1,q_2,\ldots,q_n)\in\mathbb{R}^n|\ d_{\infty}(p,q)=q_i-p_i\}, \\
S_i^{-}(p)&=&\{q=(q_1,q_2,\ldots,q_n)\in\mathbb{R}^n|\ d_{\infty}(p,q)=p_i-q_i\}
\end{eqnarray*}
for $i=1,2,\ldots,n$ and call them the sectors at the point $p$  as shown in the following
Fig.~\ref{fd1} and Fig.~\ref{fd2}. Notice that if $q$ belongs to $S_i^{\varepsilon}(p)$, $S_i^{\varepsilon}(q)\subseteq S_i^{\varepsilon}(p)$ holds, where $\varepsilon \in \{+,-\}$.
\end{definition}

\begin{figure}[h]
\begin{center}
\begin{pspicture*}(-1.5,-1)(3,3.5)
%\psset{unit=0.7}
\pspolygon*[linecolor=lightgray](0.5,1)(2.5,3)(2.5,-1)
\psline(0.5,1)(2.5,3) \psline(0.5,1)(2.5,-1)
\pspolygon*[linecolor=gray](0.5,1)(2.5,3)(-1.5,3)
\psline(0.5,1)(-1.5,3) \psline(0.5,1)(2.5,3)
\pspolygon*[linecolor=lightgray](0.5,1)(-1.5,-1)(-1.5,3)
\psline(0.5,1)(-1.5,-1) \psline(0.5,1)(-1.5,3)
\pspolygon*[linecolor=gray](0.5,1)(-1.5,-1)(2.5,-1)
\psline(0.5,1)(2.5,-1) \psline(0.5,1)(-1.5,-1)
\psline{->}(-1.5,0.5)(2.5,0.5) \uput[r](2.5,0.5){$x$}
\psline{->}(-0.25,-1.5)(-0.25,3) \uput[u](-0.25,3){$y$} \psdot(0.5,1)
\uput[l](0.5,1){$p$} \uput[d](2,1.5){$S_1^+(p)$}
\uput[u](0.7,1.8){$S_2^+(p)$} \uput[l](-0.25,1.2){$S_1^-(p)$}
\uput[d](0.75,0){$S_2^-(p)$}
\end{pspicture*}
\caption{Sectors of a point $p$ in $(\mathbb{R}^2, d_{\infty})$.}
\label{fd1}
\end{center}
\end{figure}
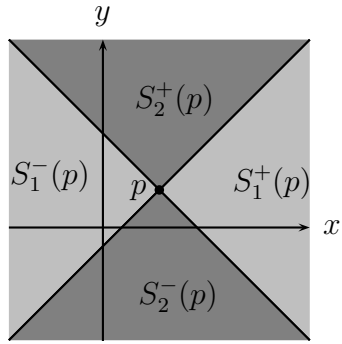

\begin{figure}[h!]
\begin{center}
\begin{pspicture*}(-2,-0.5)(8,6.5)
\psset{unit=0.85}
\pspolygon*[linecolor=gray](4,0)(3.25,2.75)(4,4)(6,6)(6,2)(4,0)
\psline(0,4)(4,4)(4,0)
\psline(0,4)(2,6)(6,6)(6,2)(4,0)
\psline(6,6)(4,4)
\psline(4,0)(0,0)(0,4)
\psline[linewidth=0.5pt, linestyle=dashed](0,0)(2,2)(6,2)
\psline[linewidth=0.5pt, linestyle=dashed](2,2)(2,6)
\psdot(3.25,2.75)
\uput[d](3.25,2.75){$O$}
\psline(3.25,2.75)(4,4)
\psline(3.25,2.75)(6,6)
\psline(3.25,2.75)(6,2)
\psline(3.25,2.75)(4,0)
\uput[r](4.2,3.25){$S_2^{+}(O)$}
\psline{->}(3.25,2.75)(2.9,2.4)
\uput[d](2.75,2.75){$x$}
\psline{->}(3.25,2.75)(3.75,2.75)
\uput[r](3.45,2.88){$y$}
\psline{->}(3.25,2.75)(3.25,3.25)
\uput[u](3.25,3.25){$z$}
\end{pspicture*}
\caption{The sector $S_2^{+
}(O)$ of the origin in $(\mathbb{R}^3,d_{\infty})$.}
\label{fd2}
\end{center}
\end{figure}
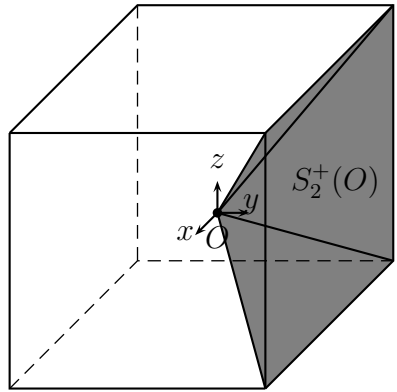

The following theorem gives a characterization of geodesics in $(\mathbb{R}^n, d_{\infty})$.
\begin{figure}[h]
\begin{center}
\begin{pspicture*}(-1.5,-2)(8,2.5)
\psline{->}(-1.5,0)(2.5,0) \uput[r](2.5,0){$x$}
\psline{->}(0,-1.5)(0,2) \uput[u](0,2){$y$}
\psline[linewidth=0.5pt, linecolor=darkgray](0.5,2)(-1,0.5)(0.5,-1)
\psline[linewidth=0.5pt, linecolor=darkgray](1.5,1.6)(0,0.1)(1.5,-1.4)
\psline[linewidth=0.5pt, linecolor=darkgray](2.5,0.9)(1,-0.6)(2.5,-2.1)
\psdot(-1,0.5)
\uput[l](-1,0.5){$p$}
\psdot(2,-0.25)
\uput[r](2,-0.25){$q$}
\pscurve(-1,0.5)(0,0.1)(1,-0.6)(2,-0.25)

\psline{->}(3.5,0)(7.5,0) \uput[r](7.5,0){$x$}
\psline{->}(5,-1.5)(5,2) \uput[u](5,2){$y$}
\psline[linewidth=0.5pt, linecolor=darkgray](5.5,2)(4,0.5)(5.5,-1)
\psline[linewidth=0.5pt, linecolor=red](6.5,1.75)(5,0.25)(6.5,-1.25)
\psdot(4,0.5)
\uput[l](4,0.5){$p$}
\psdot(7,-0.25)
\uput[r](7,-0.25){$q$}
\pscurve(4,0.5)(5,0.25)(6,1.25)(7,-0.25)
\end{pspicture*}
\caption{Two paths between $p$ and $q$ in $\mathbb{R}^2_{\infty}$ one of which (on the left) is a geodesic but the other is not.}
\label{fd3}
\end{center}
\end{figure}
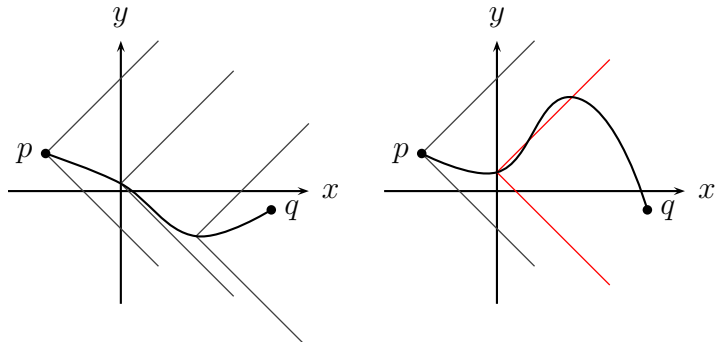
%%%%%%%%%%%%%%
\begin{proposition}[Theorem 2.2 of\cite{Mehmet}]
\label{d1}
Let $p=(p_1,p_2,\ldots,p_n)$, $q=(q_1,q_2,\ldots,q_n)\in\mathbb{R}^n$ be two points, $q\in S_i^\varepsilon(p)$ and $\alpha:[0,d(p,q)\,]\rightarrow \mathbb{R}^n$ be a  path such that $\alpha(0)=p$ and $\alpha(d(p,q))=q$. Then $\alpha$ is a geodesic in $\mathbb{R}_\infty^n$ if and only if $\alpha(t')\in S_i^\varepsilon(\alpha(t))$ for all $t,t'\in [0,d(p,q)]$ such that $t<t'$. (See Fig~\ref{fd3})
\end{proposition}
We will also need the following ``Shortening Lemma'':

\begin{lemma} \label{a1}
Let $\alpha:[0,b]\rightarrow(\mathbb{R}^n,d_{\infty})$ be a geodesic and $0<c<d<b$. Then, $\tilde{\alpha}:[0, b-d+c]\rightarrow(\mathbb{R}^n,d_{\infty})$,
 \begin{equation*}
 \tilde{\alpha}(t)=\left\{\begin{array}{cll}
 \alpha(t) &when& t\in [0,c]\\
 \alpha(t-c+d)-\alpha(d)+\alpha(c) &when& t\in [c,b-d+c]\\
 \end{array}\right.
 \end{equation*}
 is a geodesic.
\end{lemma}
\begin{proof}

Assume that $\alpha(b)\in S_i^\varepsilon(\alpha(0))$ for some $i\in\{1,2,\ldots,n\}$ and $\varepsilon\in \{+, -\}$. Since $\alpha$ is a geodesic, for all $t,t'\in [0,b]$ such that $t<t'$ we get $\alpha(t')\in S_i^\varepsilon(\alpha(t))$.  Given $t,t'\in [0,b-d+c]$ with $t<t'$. We consider three possibilities: if $t,t'\in [0,c]$,  then $\tilde{\alpha}(t')\in S_i^\varepsilon(\tilde{\alpha}(t))$ because $\tilde{\alpha}= \alpha$ on $[0,c]$. If $t,t'\in [c,b-d+c]$, then $\tilde{\alpha}(t')\in S_i^\varepsilon(\tilde{\alpha}(t))$ because $\alpha(t'-c+d)\in  S_i^\varepsilon( \alpha(t-c+d))$ and this implies $\alpha(t'-c+d)-\alpha(d)+\alpha(c)\in  S_i^\varepsilon( \alpha(t-c+d)-\alpha(d)+\alpha(c))$ . If $t\in [0,c]$ and $t'\in [c,b-d+c]$, we know that $\tilde{\alpha}(t')\in S_i^\varepsilon(\tilde{\alpha}(c))$ and $\tilde{\alpha}(c)\in S_i^\varepsilon(\tilde{\alpha}(t))$. Since $\tilde{\alpha}(c)\in S_i^\varepsilon(\tilde{\alpha}(t))$, $S_i^\varepsilon(\tilde{\alpha}(c))\subseteq S_i^\varepsilon(\tilde{\alpha}(t))$; hence, we get $\tilde{\alpha}(t')\in S_i^\varepsilon(\tilde{\alpha}(t))$. Thus, previous proposition implies that $\tilde{\alpha}$
 is a geodesic.
\end{proof}

\begin{proposition}
A finite metric tree with more than $2^n$ leaves can not be embedded isometrically into $(\mathbb{R}^n,d_{\infty})$.
\end{proposition}

\begin{proof}
Let us assume that $X$ has $k$ leaves, $k>2^n$ and $f:X\rightarrow(\mathbb{R}^n,d_{\infty})$ be an isometric embedding. Denote the leaves $a_1,a_2,\ldots,a_k$, and their images under $f$ by $A_1,A_2,\ldots,A_k$, i.e. $f(a_i)=A_i$. Let us denote the vertex points on $X$ by $b_i$ for $i=1,2,\ldots,k$ such that there exists an edge between $a_i$ and $b_i$ and assume $f(b_i)=B_i$. Note that the vector $A_iB_i$ can not be equal to $t\cdot(A_jB_j)$ for $i\neq j$ and $t>0$. Because if we assume $A_iB_i=t\cdot(A_jB_j)$ or equivalently $B_i-A_i=t\cdot(B_j-A_j)$, we get
 \begin{eqnarray*}
||A_i-A_j||&=&||(A_i-B_i)+(B_i-B_j)+(B_j-A_j)||\\
&=&||(1-t)(B_j-A_j)+(B_i-B_j)|| \\
&\leq&|1-t|\cdot||B_j-A_j||+||B_i-B_j||
\end{eqnarray*}

On the other hand, since the geodesic from $a_i$ to $a_j$ passes through $b_i$ and $b_j$ respectively, we have

\begin{eqnarray*}
  ||A_i-A_j|| &=& ||A_i-B_i||+||B_i-B_j||+||B_j-A_j)|| \\
   &=& t||B_j-A_j||+||B_i-B_j||+||B_j-A_j)|| \\
   &=& (1+t)||B_j-A_j||+||B_i-B_j||
\end{eqnarray*}
and this contradicts previous inequality.

Now consider the set $\{ t . (A_i-B_i) \, |\, 0\leq t \leq 1, \, i=1,2,\ldots,k\}$. According to Lemma~\ref{a1}, this set is a star tree with $k > 2^n$ leaves. But this contradicts the Proposition~\ref{a0}.
\end{proof}

\begin{proposition}
A finite metric tree with more than $2n$ leaves can not be embedded isometrically into $(\mathbb{R}^n,d_1)$.
\end{proposition}

Proof of this claim is very similar to the  proof above. In fact, let us assume that $X$ has $k$ leaves with $k>2n$ and $f:X\rightarrow(\mathbb{R}^n,d_1)$ be an isometric embedding. Denote those leaves by $a_1,a_2,\ldots,a_k$, and their images by $A_1,A_2,\ldots,A_k$, i.e. $f(a_i)=A_i$. Let us denote the vertices on $X$ by $b_i$ for $i=1,2,\ldots,k$ such that there is an edge between $a_i$ and $b_i$  and assume $f(b_i)=B_i$. Now consider the set $\{t\cdot(A_i-B_i)\, |\, 0\leq t\leq1,\, i=1,2,\ldots,k\}$.
%According to remark~\ref{rem1}
 This set is a star tree with $k>2n$ leaves. But this contradicts the Proposition~\ref{a20}.

\begin{theorem}\label{l_1}
Let $(X,d)$ be a metric tree. If $X$ contains at most $2n$ leaves, it can be embedded isometrically into $(\mathbb{R}^n,d_1)$.
\end{theorem}

\begin{proof}
We will prove this theorem by induction on $n$. If $n=1$, the statement is obviously true. Assume that we can embed isomerically any metric tree which has $2n$ leaves into $(\mathbb{R}^n,d_1)$. Let $X$ be any metric tree which has $2(n+1)$ leaves. We will show that $X$ can be embedded isometrically into $(\mathbb{R}^{n+1},d_1)$. We can choose two leaves in $X$ such that after deleting them with the adjacent edges (and discarding possibly emerging vertices with degree 2), the rest of the tree has $2n$ leaves (see Figure~\ref{af24}). Let us call these leaves as $A_0$ and $A_1$ and their adjacent edges as $B_0A_0$ and $B_1A_1$. According to our assumption, there is an isometric embedding $f$ from rest of the tree $Y=X-((B_0A_0]\cup(B_1A_1])$ to $(\mathbb{R}^n,d_1)$. Define the map $F:X\rightarrow (\mathbb{R}^{n+1},d_1)$,

 \begin{equation*}
 F(x)=\left\{\begin{array}{cll}
 (f(x),0) &when& x\in Y\\
 (f(B_0),-d(x,B_0)) &when& x\in[B_0A_0]\\
  (f(B_1),d(x,B_1)) &when& x\in[B_1A_1]\\
 \end{array}\right.
 \end{equation*}

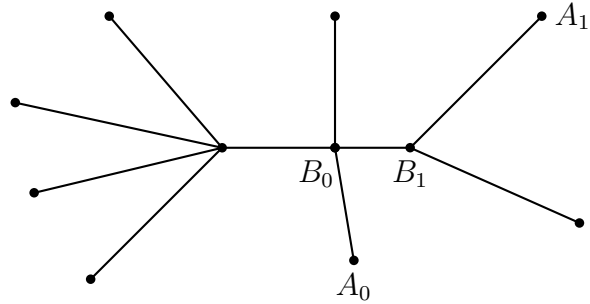
\begin{figure}[h!]
\begin{center}
\begin{pspicture*}(-4,-2.5)(4,2)
\psline(-2.5,1.75)(-1,0)(-3.75,0.6)
\psline(-2.75,-1.75)(-1,0)(-3.5,-0.6)
\psline(-1,0)(0.5,0)
\psline(0.5,1.75)(0.5,0)(0.75,-1.5)
\psline(0.5,0)(1.5,0)(3.25,1.75)
\psline(1.5,0)(3.75,-1)
\psdot(-1,0)
\uput[d](0.75,-1.5){$A_0$}
\psdot(0.5,0)
\uput[d](0.25,0){$B_0$}
\uput[d](1.5,0){$B_1$}
\uput[r](3.25,1.75){$A_1$}
\psdots(-2.5,1.75)(-3.75,0.6)(-2.75,-1.75)(-3.5,-0.6)(0.5,1.75)(0.75,-1.5)(0.5,0)(1.5,0)(3.25,1.75)(3.75,-1)
\end{pspicture*}
\caption{If we delete $(B_0A_0]$ and $(B_1A_1]$ and discarding the vertex $B_1$, we get a tree which has two fewer leaves.}
\label{af24}
\end{center}
\end{figure}

We will show that $F$ is an isometric embedding. Let $x,y\in X$ be two arbitrary  points. Since $f$ is an isometric embedding, If $x,y\in Y$, we get
\[
d_1(F(x),F(y))=d_1(f(x),f(y))=d(x,y).
\]

If $x\in (B_0A_0]$ and $y\in Y$,
\begin{eqnarray*}
  d_1(F(x),F(y)) &=& d_1(f(B_0),f(y))+d(x,B_0) \\
  &=& d(B_0,y)+d(x,B_0)\\
  &=&d(x,y).
\end{eqnarray*}

If $x\in (B_1A_1]$ and $y\in Y$,
\begin{eqnarray*}
  d_1(F(x),F(y)) &=& d_1(f(B_1),f(y))+d(x,B_1) \\
  &=& d(B_1,y)+d(x,B_1)\\
  &=&d(x,y).
\end{eqnarray*}

If $x\in (B_0A_0]$ and $y\in (B_1A_1]$,
\begin{eqnarray*}
  d_1(F(x),F(y)) &=& d_1(f(B_0),f(B_1))+d(x,B_0)+d(y,B_1) \\
  &=&d(x,B_0)+d(B_0,B_1)+d(B_1,y)\\
  &=&d(x,y).
\end{eqnarray*}

\end{proof}

\bibliographystyle{amsplain}

\end{document}